\documentclass[12pt, a4paper, leqno]{amsart}
\usepackage{amsmath}
\usepackage{amsfonts}
\usepackage{amssymb}
\usepackage{times}
\usepackage[T1]{fontenc} 
\usepackage{url} 
\usepackage{color,esint}
\usepackage{graphicx} 
\usepackage{enumitem}

\usepackage{mathtools}

\let\oldmarginpar\marginpar
\renewcommand\marginpar[1]{\-\oldmarginpar[\raggedleft\footnotesize #1]%
{\raggedright\footnotesize #1}}

\usepackage{amsthm}
\theoremstyle{plain}
\newtheorem{thm}[equation]{Theorem}
\newtheorem{lem}[equation]{Lemma}
\newtheorem{prop}[equation]{Proposition}

\theoremstyle{definition}
\newtheorem{defn}[equation]{Definition}

\theoremstyle{remark}
\newtheorem{rem}[equation]{Remark}

\numberwithin{equation}{section}

\newcommand{\R}{\mathbb{R}}

\newcommand{\Rn}{\mathbb{R}^n}

\def\essinf{\operatornamewithlimits{ess\,inf}}
\def\essliminf{\operatornamewithlimits{ess\,\lim\,inf}}
\newcommand{\esslimsup}{\operatornamewithlimits{ess\,\lim\,sup}}

\renewcommand{\phi}{\varphi}

\def\le{\leqslant}
\def\leq{\leqslant}
\def\ge{\geqslant}
\def\geq{\geqslant}
\def\phi{\varphi}
\def\rho{\varrho}
\def\vartheta{\theta}

\def\esssup{\operatornamewithlimits{ess\,sup}}

\def\diam{\qopname\relax o{diam}}

\def\spt{\qopname\relax o{spt}}
\def\dist{\qopname\relax o{dist}}

\def\phix{{\phi(\cdot)}}

\def\loc{{\rm loc}}

\date{\today}

\begin{document}

\title[Global continuity and higher integrability]{Global continuity and higher integrability of a minimizer of an obstacle problem under generalized Orlicz growth conditions}
\author{Arttu Karppinen}


\begin{abstract}
We prove continuity up to the boundary of the minimizer of an obstacle problem and higher integrability of its gradient under generalized Orlicz growth. The result recovers similar results obtained in the special cases of polynomial growth, variable exponent growth and produces new results
for Orlicz and double phase growth.
\end{abstract}

\maketitle


\section{Introduction}

We study the obstacle problem related to the Dirichlet energy integral over a bounded domain $\Omega \subset \R^n$ with boundary values in the Sobolev sense
\begin{align*}
\inf \int_{\Omega} \phi(x, |\nabla u|) \, dx,
\end{align*}
where the infimum is taken over functions $u \in W^{1,\phi}(\Omega)$ such that, given functions $\psi, f :\Omega \to \R$, we have $u \geq \psi$ almost everywhere and $u-f \in W^{1,\phi}_0(\Omega)$. In this paper we assume that $\phi$ satisfies generalized Orlicz growth conditions (see Section \ref{sec:prop}). This class of growth conditions generalize several interesting special cases such as the standard polynomial growth $t \mapsto t^p$, Orlicz growth $t \mapsto \phi(t)$, see for example \cite{BreV13}, variable exponent growth $t \mapsto t^{p(x)}$, see for example \cite{DieHHR11,Ok16} and double phase case $t \mapsto t^{p} + a(x) t^q$, see for example \cite{BarCM15}. Additionally, the problem is motivated by the study of partial differential equations, see for example \cite{GwaWWZ12}.

In this paper we prove two main results of which the first concerns the boundary continuity of a minimizer of the obstacle problem. For definitions and assumptions, see Sections \ref{sec:prop} and \ref{sec:esteongelma}. To best of our knowledge, the result is new even in the special cases of Orlicz and double phase growth.
\begin{thm}
\label{thm:main-result}
Let $\phi \in \Phi_c(\R^n)$ be strictly convex and satisfy (A0), (A1), (A1-$n$), (aInc) and (aDec). Let $\psi \in C(\Omega)$ and $f \in C(\overline{\Omega})\cap W^{1,\phi}(\Omega)$ be such that $\mathcal{K}_{\psi}^{f}(\Omega) \not = \emptyset$ and let $u$ be the continuous minimizer of the $\mathcal{K}^{f}_{\psi}(\Omega)$-obstacle problem from Theorem \ref{thm:u-is-continuous}. If $x_0 \in \partial \Omega$ satisfies the capacity fatness condition \eqref{eq:capacity-fatness}, then
\begin{align*}
\lim_{x \to x_0} u(x) = f(x_0).
\end{align*}
\end{thm}
\noindent A similar result in the generalized Orlicz setting without the obstacle has been proven in \cite{HarH_pp16}. 

The proof for interior continuity follows outlines given in the book of Bj\"orns' \cite{BjoB11}. The proof of the main theorem and few intermediate results are analogous to \cite{HeiKM06}, since scaling of the minimizer does not preserve minimality in the generalized Orlicz case. We also the study relationship of the measure density condition and $\phi$-fatness: the former implies the latter when $q<n$ (Lemma \ref{lem:dens-fat}). For further information about capacities in this context, see for example \cite{BarHH_pp17,HarH_pp16} in $\R^n$ and \cite{OhnS16} in metric measure spaces. 

The second main result is global higher integrability of the gradient:
\begin{thm}[Global higher integrability of the gradient]
\label{thm:main-theorem}
Suppose that $\phi \in \Phi_w(\R^n)$ satisfies conditions (A0), (A1), (aInc) and (aDec). Additionally suppose that the measure density condition \eqref{eq:measure-density} is fulfilled at every point $x_0 \in \partial \Omega$ with a constant $c_\ast$, and let $u$ be the minimizer of the $\mathcal{K}^{f}_{\psi}(\Omega)$-obstacle problem, where $\psi, f \in W^{1,\phi}(\Omega)$ and $\phi(\cdot, |\nabla \psi|), \phi(x, |\nabla f|) \in L^{1+\delta}(\Omega)$ for some
$\delta > 0$ and $\mathcal{K}^{f}_{\psi}(\Omega) \not = \emptyset$.
Then there exist $\varepsilon > 0$ and a constant $C=C(n,\phi,c_\ast)$ such that $\phi(x, |\nabla u|) \in L^{1+\varepsilon}(\Omega)$ and
\begin{align}
\label{eq:main-estimate}
\begin{split}
\int_{\Omega} \phi(x, |\nabla u|)^{1+\varepsilon} \, dx &\leq C \bigg [ \left ( \int_{\Omega} \phi(x, |\nabla u|) \, dx\right )^{1+\varepsilon} \\
&\quad + \int_{\Omega} \phi(x, |\nabla \psi|)^{1+ \varepsilon} \, dx + \int_{\Omega} \phi(x, |\nabla f|)^{1+\varepsilon} \, dx + 1 \bigg].
\end{split}
\end{align}
\end{thm}
This result continues the recently published article \cite{HarHK18}, where the authors proved local higher integrability of the gradient of the quasiminimizer. Now the result is improved to a global result and the problem is generalized with an obstacle $\psi$. These results are steps towards higher regularity results of the minimizer such as H\"older continuity for every exponent $\beta \in (0,1)$ and H\"older continuity of the gradient. For example in \cite{HasO19} local higher integrability of the gradient is used several times in the proof.
Again, to best of our knowledge, produces new results in special cases of Orlicz and double phase growth. For variable exponent analogue, see \cite{EkeHL13}.

The strategy of the proof is to combine two Caccioppoli inequalities with the previously
proven Sobolev--Poincar\'e inequality to lay ground for Gehring's lemma. The first Caccioppoli
inequality handles the interior case with the obstacle and the second inequality handles balls nearly overlapping with the boundary of $\Omega$. 
To achieve global results in general we assume that the measure density condition \eqref{eq:measure-density} is fulfilled at every boundary point.

\section{Properties of generalized $\Phi$-functions}
\label{sec:prop}

By $\Omega \subset \Rn$ we denote a bounded domain, i.e.\ a bounded, open and connected set.
When $A$ and $B$ are open sets and $\overline{A}$ is compact, by $A \Subset B$ we mean that $\overline{A} \subset B$. The measure of a set $A$ is denoted by $|A|$.
By $c$ or $C$ we denote a generic constant whose
value may change between appearances.
A function $f$ is \textit{almost increasing} if there
exists a constant $L \ge 1$ such that $f(s) \le L f(t)$ for all $s \le t$
(more precisely, $L$-almost increasing).
\textit{Almost decreasing} is defined analogously. A function $f$ is called \emph{convex} if $f(tx + (1-t)y) \leq tf(x) + (1-t)f(y)$ for every $t \in (0,1)$. Strict convexity assumes that the previous inequality is strict.

\begin{defn}
We say that $\phi: \Omega\times [0, \infty) \to [0, \infty]$ is a
\textit{weak $\Phi$-function}, and write $\phi \in \Phi_w(\Omega)$, if
\begin{itemize}
\item 
For every $t \in [0, \infty)$ the function $x \mapsto \phi(x, t)$ is measurable and for every $x \in \Omega$ the function $t \mapsto \phi(x, t)$ is increasing.
\item 
$\displaystyle \phi(x, 0) = \lim_{t \to 0^+} \phi(x,t) =0$ and $\displaystyle \lim_{t \to \infty}\phi(x,t)=\infty$ for every $x\in \Omega$.
\item 
The function $t \mapsto \frac{\phi(x, t)}t$ is
$L$-almost increasing for $t>0$ and every $x\in \Omega$. 
\item 
The function $t \mapsto \phi(x,t)$ is left-continuous for $t >0$ and every $x \in \Omega$.
\end{itemize}
If, additionally, $t \mapsto \phi(x,t)$ is convex, we denote $\phi \in \Phi_c(\Omega)$ and say that $\phi$ is a convex $\Phi$-function.
\end{defn}

By $\phi^{-1}$ we mean the \emph{left inverse} of $\phi$, defined as
\begin{align*}
\phi^{-1}(\tau):= \inf \{t \geq 0 \, : \, \phi(t)\geq \tau\}.
\end{align*}
Let us write $\phi^+_B (t) := \sup_{x \in B\cap \Omega} \phi(x, t)$ and
$\phi^-_B (t) := \inf_{x \in B\cap \Omega} \phi(x, t)$;
and abbreviate $\phi^\pm := \phi^\pm_\Omega$.
Throughout the paper we need one or multiple of the following assumptions.
\begin{itemize}[leftmargin=1.7cm]
\item[(A0)]
There exists $\beta \in(0,1)$ such that $\phi^+(\beta) \le 1 \le \phi^-(1/\beta)$.
\item[(A1)]
There exists $\beta\in (0,1)$ such that, for every ball $B \cap \Omega \not = \emptyset$,
\[
\phi^+_B (\beta t) \le \phi^-_B (t) \quad \text{when }  t \in \left [1, (\phi_B^-)^{-1}\left (\tfrac1{|B|} \right ) \right ]
\]
\item[(A1-$n$)]
There exists $\beta \in (0,1)$ such that, for every ball $B \cap \Omega \not = \emptyset$,
\[
\phi^{+}_{B}(\beta t) \leq \phi_{B}^{-}(t) \quad \text{when } t \in \left [1, \tfrac{1}{\diam(B)}\right ].
\]

\end{itemize}

We also introduce the following assumptions, which are of different nature. They are 
related to the $\Delta_2$ and $\nabla_2$ conditions from Orlicz space theory. 
\begin{itemize}[leftmargin=1.7cm]
\item[(aInc)$_p$]
There exists $L\ge 1$ such that $t \mapsto \frac{\phi(x,t)}{t^{p}} $ is $L$-almost increasing in $(0,\infty)$.
\item[(aDec)$_q$]
There exists $L\ge 1$ such that $t \mapsto \frac{\phi(x,t)}{t^{q}} $ is $L$-almost decreasing in $(0,\infty)$.
\end{itemize}
We write (aInc) if there exists $p>1$ such that (aInc)$_p$ holds, similarly for (aDec). For brevity, we may write for example that a constant $C = C(n, \phi)$, in which case $C$ depends on the dimension and some or all of the parameters listed in the previous assumptions related to $\phi$. 

Despite the technical formulation of the assumptions, each of them has an intuitive interpretation. (A0) declares the space to be unweighed, (A1) is a continuity assumption with respect to the space variable, while (A1-$n$) takes account of the dimension also. These are generalizations of the $\log$-H\"older continuity of the variable exponent spaces and the assumption $\frac{q}{p} \leq 1+ \frac{\alpha}{n}$ of the double phase case. Lastly, (aInc)$_p$ and (aDec)$_q$ state that globally $\phi(x,t)$ grows faster than $t^p$ and slower than $t^q$.

We say that $\phi$ is \textit{doubling} if there exists a constant $L\ge 1$ such that $\phi(x, 2t) \le L \phi(x, t)$
for every $x \in \Omega$ and every $t \geq 0$. If $\phi$ is doubling with constant $L$, then
by iteration
\begin{equation}\label{equ:doubling_iteration}
\phi(x,t) \le L^2 \Big( \frac{t}{s}\Big)^{Q} \phi(x,s)
\end{equation}
for every $x \in \Omega$ and every $0<s<t$, where $Q= \log_2(L)$. For the proof see for example \cite[Lemma 3.3, p.~66]{BjoB11}. 
Note that doubling also yields that
\begin{align}
\label{eq:phi-triangle}
\phi(x, t+s)\le L \phi(x, t) + L \phi(x, s).
\end{align}
Since (aDec) is equivalent to doubling \cite[Lemma 2.6]{HarHT17}, inequality \eqref{eq:phi-triangle} holds for $\phi$ satisfying (aDec). In the proofs we often use the phrase like ''using (aDec)'' and mean doubling or its consequence \eqref{eq:phi-triangle}.

Generalized Orlicz and Orlicz--Sobolev spaces have been studied with our assumptions for example in \cite{HarH_pp16,HarH17,HarH18,HarHK16,HarHK18,HarHT17}. We recall some definitions. We denote by $L^0(\Omega)$ the set of measurable functions in $\Omega$ and the integral average of a function $f$ over a set $A$ is denoted by $\fint_A f(x) \, dx =: f_A$. 
Additionally, we denote the positive and negative part of a function as $f_+ = \max\{f,0\}$ and $f_- = \max\{-f,0\}$.
\begin{defn}
Let $\phi \in \Phi_w(\Omega)$ and define the \emph{modular} $\varrho_{\phi}$ for $f \in L^0(\Omega)$ by
\begin{align*}
\varrho_{\phi} (f):=\int_\Omega \phi(x,|f(x)|) \, dx.
\end{align*}
The \emph{generalized Orlicz space}, also called Musielak--Orlicz space, is defined as the set
\begin{align*}
L^\phi(\Omega) := \left \lbrace f \in L^0(\Omega): \, \varrho_{\phi}(\lambda f) <\infty \text{ for some } \lambda >0 \right  \rbrace
\end{align*}
equipped with the (Luxemburg) norm 
\begin{align*}
\|f\|_{L^\phi(\Omega)} := \inf \left\lbrace \lambda >0: \, \varrho_{\phi} \left( \frac{f}{\lambda} \right ) \leq 1 \right\rbrace.
\end{align*}
If the set is clear from the context we abbreviate $\|f\|_{L^\phi(\Omega)}$ by $\|f\|_{\phi}$. A function $f$ belongs to \emph{local generalized Orlicz space} $L^{\phi}_{\loc}(\Omega)$ if $\|f\|_{L^\phi(K)} < \infty$ for every compact set $K \Subset \Omega$.

A function $u \in L^\phi(\Omega)$ belongs to the \emph{generalized Orlicz--Sobolev space} $W^{1,\phi}(\Omega)$ if its weak partial derivatives $\partial_1 u, \dots, \partial_n u$ exist and belong to $L^\phi(\Omega)$. The norm of Orlicz--Sobolev space is defined as $\|f\|_{W^{1,\phi}(\Omega)} := \|f\|_{L^\phi(\Omega)} + \|\nabla f\|_{L^\phi(\Omega)}$, where $\nabla f$ is the weak gradient of $f$. 
Additionally we define $W^{1,\phi}_0(\Omega)$ as the closure of the space $C_0^{\infty}(\Omega)$ with respect to the norm of Orlicz--Sobolev space.
\end{defn}

The definition of $W^{1,\phi}_0(\Omega)$ is reasonable, as $C_0^\infty(\Omega)$ is dense in $W^{1,\phi}_0(\Omega)$ if $\phi$ satisfies (A0), (A1) and (aDec) and $\Omega$ is bounded \cite[Theorem 6.4.6]{HarH18} (boundedness of $\Omega$ frees us of the assumption (A2)). 

The modular $\varrho_\phi$ and the norm have the following useful property, called the unit ball property \cite[Lemma 3.2.5]{HarH18}. However, in our case we need only the following implication which follows from the definition of the norm 
\begin{align}
\label{eq:unit-ball-property}
\varrho_{\phi}(f)\leq 1 \quad \Rightarrow \quad \|f\|_{\phi} \leq 1.
\end{align}

Next we recall the definition of relative Sobolev capacity of a set as a another way to measure the size of a set. Basic properties of this capacity have been studied in \cite{BarHH_pp17}.
\begin{defn}
\label{def:capacity}
Let $\phi \in \Phi_w(\Omega)$ and $E \Subset \Omega$. Then \emph{relative Sobolev capacity of} $E$ is defined as
\begin{align*}
C_{\phi}(E, \Omega) = \inf_{u \in S_{\phi}(E,\Omega)} \int_{\Omega} \phi(x, |\nabla u|) \, dx,
\end{align*}
where the infimum is taken over the set $S_{\phi}(E,\Omega)$ of all functions $u \in W^{1,\phi}_0(\Omega)$ with $u \geq 1$ in an open set containing $E$.
\end{defn}

In order to attain global results, some regularity of the boundary has to be assumed. In this paper we use the measure density and capacity fatness conditions 
\begin{align}
\label{eq:measure-density}
|B(x_0,r) \setminus \Omega| \geq c_\ast |B(x_0,r)|
\end{align}
\begin{align}
\label{eq:capacity-fatness}
C_\phi (B(x_0,r) \setminus \Omega, B(x_0, 2r)) \geq c_\ast \,  C_\phi (B(x_0,r), B(x_0, 2r)),
\end{align}
where $B(x_0,r)$ is a ball centred at a point $x_0 \in \partial \Omega$ and $r \leq R$ for some $R>0$ and $c_\ast \in (0,1)$.
 The measure density condition is often sufficiently general as for example all domains with Lipschitz boundary satisfy it and therefore it is commonly used in regularity theory. However the capacity fatness condition was used in \cite{HarH_pp16} so we get the more general result with ease in the case of boundary continuity.

Even though we consider minimizing problem in $\Omega$ we assume that $\phi$ is defined in the whole $\R^n$ since later we need to consider the complement of $\Omega$ due to previous boundary conditions.

\section{Auxiliary results}
Let us first collect some general lemmas, which are not related to the obstacle problem directly.
First, we state the following lemma \cite[Lemma 2.11]{HarH_pp16}, to which we refer to throughout the paper.

\begin{lem}
\label{lem:min-zero-valued}
Let $\Omega \subset \R^n$ be bounded. Let $\phi \in \Phi_w(\Omega)$ satisfy (A0), (A1) and (aDec). If $v \in W^{1,\phi}(\Omega)$ is non-negative and $u \in W^{1,\phi}_0(\Omega)$, then $\min\{u,v\} \in W^{1,\phi}_0(\Omega)$.
\end{lem}

The next lemma is intuitively clear, and follows easily from the previous lemma.
\begin{lem}
\label{lem:u-between-compacts}
Let $\phi \in \Phi_w(\Omega)$ satisfy (A0), (A1) and (aDec). Let $u \in W^{1,\phi}(\Omega)$ and $v,w \in W^{1,\phi}_0(\Omega)$. If $v \leq u \leq w$ almost everywhere in $\Omega$, then $u \in W^{1,\phi}_0(\Omega)$.
\end{lem}
\begin{proof}
If we subtract $v$ from all the terms in the inequality and notice that $u-v \in W^{1,\phi}_0(\Omega)$ if and only if $u \in W^{1,\phi}_0(\Omega)$, we can assume that $v=0$ almost everywhere in $\Omega$. Now since $u$ is non-negative and $w \in W^{1,\phi}_0(\Omega)$, Lemma \ref{lem:min-zero-valued} implies that $u= \min\{u,w\} \in W^{1,\phi}_0(\Omega)$.  
\end{proof}

Lastly, we prove a lemma regarding sequences of maxima and minima which is important when we are comparing functions pointwise or handling just the positive part of a function. The restriction to subsequences is not severe since later on we need the existence of a sequence rather than convergence of a specific sequence.
\begin{lem}
\label{lem:min-max-conv}
Let $\phi \in \Phi_w(\Omega)$ satisfy (A0), and (aDec). If $u_j, v_j \in W^{1,\phi}(\Omega)$ converge to $u$ and $v$ respectively in $W^{1,\phi}(\Omega)$, then there are subsequences such that $\min\{u_{j_k},v_{j_k}\} \to \min\{u,v\}$ and $\max\{u_{j_k},v_{j_k}\} \to \max\{u,v\}$ in $W^{1,\phi}(\Omega)$.
\end{lem}

\begin{proof}
Because, for example, $\min\{f,g\} = g + \min\{f-g,0\}$, it suffices to show that if $u_j$ converges to $u$ in $W^{1,\phi}(\Omega)$, then $\left ((u_{j_k})_+\right )$ converges to $u_+$, where $(u_{j_k})$ is the pointwise converging subsequence. This subsequence exists because assumption (A0) and (aInc)$_1$ imply that $W^{1,\phi}(\Omega) \subset W^{1,1}(\Omega)$ \cite[Lemma 4.4]{HarHK16}. Since $|(u_{j_k})_+ -u_+| \leq |u_{j_k}-u|$, $t \mapsto \phi(\cdot, t)$ is increasing and norm convergence is equivalent to modular convergence when $\phi$ satisfies (aDec) \cite[Corollary 3.3.4]{HarH18}, we get
\begin{align*}
\int_{\Omega} \phi(x,|(u_{j_k})_+ - u_+|) \, dx \leq \int_{\Omega} \phi(x,|u_{j_k}-u|) \, dx \to 0
\end{align*}
as $j_k \to \infty$.

As for the gradients, using \eqref{eq:phi-triangle}
\begin{align*}
\int_{\Omega} &\phi(x, |\nabla (u_{j_k})_+ - \nabla u_+|)\, dx = \int_{\Omega} \phi(x, |\chi_{(0,\infty)}(u_{j_k}) \nabla u_{j_k} - \chi_{(0, \infty)}(u) \nabla u|) \, dx \\
&\leq \int_{\Omega} \phi(x, |\chi_{(0,\infty)}(u_{j_k}) \nabla u_{j_k} - \chi_{(0, \infty)}(u) \nabla u + \chi_{(0,\infty)}(u_{j_k})\nabla u - \chi_{(0,\infty)}(u_{j_k})\nabla u|) \, dx \\
&\leq L \int_{\Omega} \phi(x, |\nabla u| |\chi_{(0,\infty)} (u_{j_k}) -\chi_{(0,\infty)}(u)|) \, dx + L \int_{\Omega} \phi(x, \chi_{(0,\infty)}(u_{j_k}) |\nabla u_{j_k} - \nabla u|) \, dx \\
&\leq L \int_{\Omega} \phi(x, |\nabla u| |\chi_{(0,\infty)} (u_{j_k}) -\chi_{(0,\infty)}(u)|) \, dx + L \int_{\Omega} \phi(x,|\nabla u_{j_k} - \nabla u|) \, dx \\
& \to 0,
\end{align*}
as the first integral converges by dominated convergence \cite[Theorem 4.1]{HarHK16} ((aDec) takes care of extra assumption that $\varrho_\phi (\lambda g) < \infty$ for the dominating function $g=|\nabla u|$ and if $u\equiv 0$ in some subset of $\Omega$, then so is $|\nabla u|$) and the second integral convergences by assumption.
\end{proof}

The proof of the following Jensen type inequality can be found for example in \cite{HarHK18,Has15,Has16}. Here we have chosen $p = 1$ and simplified the assumptions on $f$ as we do not need the sharp result. Note that if $\phi$ satisfies (aDec), then the constant $\beta_0$ can be transferred to the right-hand side as a constant $C$. 

\begin{lem}
\label{lem:jensen}
Let $\phi \in \Phi_w(B)$ satisfy assumptions (A0) and (A1). There exists $\beta_0 >0$ such that
\begin{align*}
\phi \left (x, \beta_0 \fint_{B}|f| \, dy \right ) \leq \fint_{B} \phi(y,f) \, dy + 1,
\end{align*}
for every ball $B$ and $f \in L^{\phi}(B)$ with $\|f\|_{L^{\phi}(B)}\leq 1$.
\end{lem}

The proof of next proposition can be found in \cite[Proposition 6.3.13]{HarH18} and is the local version of the Sobolev--Poincar\'e inequality. One of the main ingredients in proving Theorem 1.1 is to use this inequality also with balls that overlap the complement of $\Omega$. As with the Jensen's inequality, the constant $\beta_1$ can be transferred to the right-hand side as $C$ with (aDec).

\begin{prop}[Sobolev--Poincar\'e inequality]
 \label{pro:Poincare_Sebastian_s}
 Let $\phi^{1/s} \in \Phi_w(B)$ satisfy assumptions (A0) and (A1) and let $s \in \left [1, \tfrac{n}{n-1}\right  )$. Then there exists a constant $\beta_1=\beta_1(n,s,\phi)$ 
such that
\begin{align*}
\fint_{B} \phi\bigg (x, \beta_1 \frac{|v-v_{B}|}{\diam(B)} \bigg)\,dx \leq \bigg(\fint_{B}\phi(x, |\nabla v|)^{\frac1s}\,dx \bigg )^s + 1
\end{align*}
for every $v\in W^{1,1}(B)$ with $\|\nabla v\|_{\phi^{1/s}} \leq 1$. 
\end{prop}

The following is a classical iteration lemma. For the proof, see for example  \cite[Lemma 4.2]{HarHT17}.

\begin{lem}
\label{lem:iteration}
Let $Z$ be a bounded non-negative function in the interval $[r, R] \subset \R$ and let $X: [0, \infty) \to \R$ be an increasing function which is doubling. Assume that there exists $\theta \in [0,1)$ such that
\begin{align*}
Z(s) \leq X(\tfrac{1}{t-s}) + \theta Z(t) 
\end{align*}
for all $r \leq s < t \leq R$. Then
\begin{align*}
Z(r) \lesssim X(\tfrac{1}{R-r}),
\end{align*}
where the implicit constant depends only on the doubling constant and $\theta$.
\end{lem}

The following form of Gehring's lemma can be found from \cite[Theorem 6.6 and Corollary 6.1]{Giusti}.
\begin{lem}[Gehring's lemma]\label{lem:Gehring}
Let $f \in L^1(B)$ be non-negative. Assume that $g\in L^q(4B)$ for some $q>1$ 
and that there exists $s \in(0, 1)$ such that 
\[
\fint_{B} f \, dx \lesssim \bigg( \fint_{3B} f^s \, dx\bigg)^{\frac1s} + \fint_{3B} g \, dx
\]
for every ball $B$. 
Then there exists $t>1$ such that
\[
\bigg(\fint_{B} f^t \, dx \bigg)^{\frac1t} \lesssim \fint_{4B} f \, dx + \left (\fint_{4B} g^t \, dx\right ) ^{\frac{1}{t}}.
\]
\end{lem}

\section{Properties of local minimizers and local superminimizers}
\label{sec:esteongelma}

In this paper we do not only cover (local) minimizers but also the minimizer of the so called obstacle problem. Since minimizers of obstacle problems and local superminimizers are closely related, we collect basic results regarding local superminimizers also.

\begin{defn}
Let $\psi : \Omega \to  [-\infty , \infty )$ be a function, called \emph{obstacle}, and let $f \in W^{1,\phi}(\Omega)$ be a function, which assigns the boundary values. We define admissible functions for the obstacle problem as a set
\begin{align*}
\mathcal{K}^{f}_{\psi}(\Omega) := \{ u \in W^{1, \phi}(\Omega) : u \geq \psi \text{ a.e. in } \Omega, u-f \in W^{1,\phi}_0(\Omega) \}.
\end{align*}
Additionally, we say that a function $u \in \mathcal{K}^{f}_{\psi}(\Omega)$ is a \emph{minimizer of the $\mathcal{K}^{f}_{\psi}(\Omega)$-obstacle problem} if
\begin{align*}
\int_{\Omega} \phi(x,|\nabla u|) \, dx \leq  \int_{\Omega} \phi(x,|\nabla v|) \, dx
\end{align*}
for all $v \in K^{f}_\psi(\Omega)$.
\end{defn}
If $u$ is a minimizer of the $\mathcal{K}^{f}_{-\infty}(\Omega)$-obstacle solution, we call it a \emph{minimizer} in $\Omega$.
\begin{defn}
\label{defn:local-min}
Let $\phi \in \Phi_w(\Omega)$. A function $u \in W^{1,\phi}_\loc (\Omega)$ is a \emph{local minimizer} of the $\phi$-energy in $\Omega$ if
\begin{align*}
\int_{\{v \not = 0\}} \phi(x, |\nabla u|) \, dx \leq \int_{\{v \not = 0\}} \phi(x, |\nabla (u+v)|) \, dx
\end{align*}
for all $v \in W^{1, \phi}(\Omega)$ with $\spt v \subset \Omega$, where $\spt v$ is the smallest closed set such that $v$ is non-zero almost everywhere in that set.

If the inequality is assumed only for all nonnegative or nonpositive $v$, then $u$ is called a \emph{local superminimizer} or \emph{local subminimizer}, respectively.
\end{defn}

The next lemma shows that we can often assume the test function $v$ to be pointwise bounded.
\begin{lem}
\label{lem:bdd-local-min}
Let $\phi \in \Phi_w(\Omega)$ satisfy (aDec). If $u \in W^{1,\phi}_{\loc}(\Omega)$ satisfies
\begin{align*}
\int_{\{v \not=0\}} \phi(x, |\nabla u|) \, dx \leq \int_{\{v \not =0\}} \phi(x,|\nabla(u+v)|) \, dx
\end{align*}
for all bounded $v \in W^{1,\phi}(\Omega)$ with $\spt v \subset \Omega$, then $u$ is a local minimizer of the $\phi$-energy in $\Omega$.
\end{lem}

\begin{proof}
Since $\phi$ satisfies (aDec), bounded Sobolev functions are dense in $W^{1,\phi}(\Omega)$ \cite[Lemma 6.4.2]{HarH18}. From the proof we see that if $v \in W^{1,\phi}(\Omega)$, then truncations of $v$ at level $k$, $v_k = \max\{ \min\{v(x), k\}, -k\}$, converge to $v$ in $W^{1,\phi}(\Omega)$. Additionally, $\spt v_k = \spt v$. Therefore, let $v$ be as in Definition \ref{defn:local-min} and $v_k$ be its truncations. Then, as $u$ is assumed to be a local minimizer when tested with bounded Sobolev functions with compact support and convergence in modular and norm are equivalent as $\phi$ satisfies (aDec), we get
\begin{align*}
\int_{\{v \not =0\}}\phi(x, |\nabla u|) \, dx = \int_{\{v_k \not =0\}} \phi(x, |\nabla u|) \, dx &\leq \int_{\{v_k \not =0\}} \phi(x, |\nabla(u+v_k)|) \, dx \\
&= \int_{\{v \not =0\}} \phi(x, |\nabla(u+v_k)|) \, dx.
\end{align*}
Next, as $v_k$ is a truncation, we split the integration domain accordingly
\begin{align*}
\int_{\{v \not =0\}} \phi(x, |\nabla(u+v_k)|) \, dx &= \int_{\{v \not =0\}} \chi_{\{|v| \leq k\} }\phi(x, |\nabla(u+v)|) \, dx  \\
& \quad \quad + \int_{\{v \not =0\}} \chi_{\{|v| > k\}}\phi(x, |\nabla(u+k)|) \, dx \\
&= \int_{\{v \not =0\}} \chi_{\{|v| \leq k\} }\phi(x, |\nabla(u+v)|) \, dx \\
&\quad \quad  + \int_{\{v \not =0\}} \chi_{\{|v| > k\}} \phi(x, |\nabla u|) \, dx \\
&\to \int_{\{v\not=0\}} \phi(x, |\nabla (u+v)|) \, dx
\end{align*}
by Lebesgue's monotone converge theorem for increasing and decreasing sequences and the fact that every integral is finite.
Thus combining two previous displays, we see that $u$ is a local minimizer of the $\phi$-energy in $\Omega$.
\end{proof}

Next we give a suitably general condition for non-emptiness of $\mathcal{K}^{f}_{\psi}(\Omega)$ and flexibility for the boundary function $f$. We then show that being a minimizer of the obstacle problem is a local property with suitable boundary values. For the rest of the paper we implicitly assume that $\mathcal{K}^{f}_{\psi}(\Omega)$ is non-empty.

\begin{prop}
\label{prop:K-nonempty}
Let $\phi \in \Phi_w(\Omega)$ satisfy (A0), (A1) and (aDec) and let $f, \psi \in W^{1,\phi}(\Omega)$. Then $\mathcal{K}_{\psi}^{f}(\Omega) \not = \emptyset$ if and only if $(\psi-f)_{+} \in W^{1,\phi}_0(\Omega)$.
\end{prop}

\begin{proof}
Suppose first that $u \in \mathcal{K}_{\psi}^{f}(\Omega)$. Then by Lemma \ref{lem:min-zero-valued} we see that
\begin{align*}
0 \leq (\psi-f)_{+} \leq (u-f)_+ = - \min\{-(u-f),0\} \in W^{1,\phi}_0(\Omega).
\end{align*}
Now the conclusion follows from Lemma \ref{lem:u-between-compacts}.

Suppose then that $(\psi-f)_{+}\in W^{1,\phi}_0(\Omega)$ and define $u:=\max\{\psi, f\} \in W^{1,\phi}(\Omega)$. Now
\begin{align*}
u-f = \max\{\psi-f, 0\} = (\psi-f)_{+} \in W^{1,\phi}_0(\Omega) \quad \text{and} \quad u\geq \psi \quad \text{in } \Omega.
\end{align*}
Therefore $u \in \mathcal{K}^{f}_{\psi}(\Omega)$.
\end{proof}

As $f$ matters essentially only in the boundary, it can be modified inside $\Omega$. This is useful, as for technical reasons we would like $f$ to be above the obstacle in $\Omega$.

\begin{lem}
\label{lem:f-above-psi}
Let $\phi \in \Phi_w(\Omega)$ satisfy (A0) and (A1) and suppose that $u \in \mathcal{K}_{\psi}^{f}(\Omega)$. Then $u \in \mathcal{K}_{\psi}^{\tilde f}(\Omega)$, where $f \geq \psi$ almoset everywhere in $\Omega$.
\end{lem}

\begin{proof}
Define $\tilde f := \max\{f, \psi\}$. First, we notice that $\tilde f =(\psi-f)_+ +  f$. Second, from Lemma \ref{lem:min-zero-valued} we deduce
\begin{align*}
0 \leq (\psi-f)_+ \leq \max\{u-f,0\} = - \min\{f-u,0\} \in W^{1,\phi}_0(\Omega).
\end{align*}
Now, as $\phi$ satisfies (A0) and (A1), Lemma \ref{lem:u-between-compacts} implies that $(\psi-f)_+ \in W^{1,\phi}_0(\Omega)$ and it is clear that $u-\tilde f \in W^{1,\phi}_0(\Omega)$. Thus we can use $\tilde f$ instead of $f$ as the function assigning boundary values.
\end{proof}

\begin{lem}
\label{lem:min-local}
Let $\phi \in \Phi_w(\Omega)$. Then a function $u \in W^{1,\phi}(\Omega)$ is a minimizer of the $\mathcal{K}^{u}_{\psi}(\Omega)$-obstacle problem if and only if $u$ is a minimizer of the $\mathcal{K}^{u}_{\psi}(D)$-obstacle problem for every open $D \subset \Omega$.
\end{lem}
\begin{proof}
Let us first suppose that $u$ is a minimizer of the $\mathcal{K}^{u}_{\psi}(\Omega)$-obstacle problem. Let $v \in \mathcal{K}^{u}_{\psi}(D)$.
Since $u-v \in W^{1,\phi}_0(D)$, there exist functions $\eta_j \in C^{\infty}_0(D)$ such that $\eta_j \to u-v \in W^{1,\phi}_0(D)$. By a zero extension we see that $\eta_j \in C^{\infty}_0(\Omega)$ for every $j$, which implies that $u-v$ has a zero extension to $W^{1,\phi}_0(\Omega)$, denoted by $h$. Now we can define
\begin{align*}
\tilde u(x) := 
\begin{cases}
u(x) \quad \text{ if } x \in \Omega \setminus D \\
v(x) \quad \text{ if } x \in D
\end{cases}
\end{align*}
which belongs to $\mathcal{K}^{u}_{\psi}(\Omega)$ as $\tilde u = u-h \in W^{1,\phi}(\Omega)$ and it has the correct boundary values in the Sobolev sense. Now, because $u$ is a minimizer of the $\mathcal{K}^{u}_{\psi}(\Omega)$-obstacle problem, we get
\begin{align*}
\int_{\Omega} \phi(x, |\nabla u|) \, dx \leq \int_{\Omega} \phi(x, |\nabla \tilde u|) \, dx = \int_{D} \phi(x, |\nabla v|) \, dx + \int_{\Omega \setminus D} \phi (x, |\nabla u|) \, dx.
\end{align*}
After subtracting $\int_{\Omega \setminus D}\phi(x, |\nabla u|) \, dx$ from both sides we see that $u$ is also a minimizer of the $\mathcal{K}^{u}_{\psi}(D)$-obstacle problem.

The other direction follows immediately by choosing $D=\Omega$.
\end{proof}

We recall that a solution of the $\mathcal{K}^{f}_{\psi}(\Omega)$-obstacle problem is a local superminimizer \cite[Proposition 7.16]{BjoB11} and an opposite relation also holds. 

\begin{prop}
\label{prop:obs-is-super}
Let $\phi \in \Phi_w(\Omega)$. Then a function $u$ is a local superminimizer in $\Omega$ if and only if $u$ is a minimizer of $\mathcal{K}^{u}_{u}(\Omega')$-obstacle problem for every open $\Omega' \Subset \Omega$.
\end{prop}

\begin{proof}
Suppose first that $u$ is a local superminimizer in $\Omega$. Since $\Omega' \Subset \Omega$ we have $u \in W^{1,\phi}(\Omega')$ and therefore $u \in \mathcal{K}^{u}_{u}(\Omega')$. Now, let $v \in \mathcal{K}^{u}_{u}(\Omega')$ be arbitrary and denote $w:= \max\{u,v\}$. Clearly $w=v$ almost everywhere in $\Omega'$ and thus $\eta := w-u \in W^{1,\phi}_0(\Omega')$ is nonnegative. Now, we use the local superminimality of $u$ in the set $\{\eta\not=0\}$ and the fact that $\nabla \eta =0$ almost everywhere in the set $\{\eta=0\}$ to get
\begin{align*}
\int_{\Omega'} \phi(x, |\nabla u|) \, dx &= \int_{\{\eta=0\}} \phi(x, |\nabla u|) \, dx +\int_{\{\eta \not=0\}} \phi(x, |\nabla u|) \, dx  \\
&\leq \int_{\{\eta=0\}} \phi(x, |\nabla (u+\eta)|) \, dx +\int_{\{\eta \not=0\}} \phi(x, |\nabla (u+\eta)|) \, dx \\
&=\int_{\Omega'} \phi(x, |\nabla w|) \, dx =\int_{\Omega'} \phi(x, |\nabla v|) \, dx.
\end{align*}
So $u$ is a minimizer of a $\mathcal{K}^{u}_{u}(\Omega')$-obstacle problem.

Now suppose that $u$ is a minimizer of a $\mathcal{K}^{u}_{u}(\Omega')$-obstacle problem for every $\Omega' \Subset \Omega$. Let $v \in W^{1,\phi}(\Omega)$ be nonnegative such that $\{v >0 \} \Subset \Omega$ and let $\Omega'$ be an open set such that $\{v >0\} \subset \Omega' \Subset \Omega$.
Therefore $v$ is an admissible test function for local superminimizers. As $u$ is a minimizer of the $\mathcal{K}^{u}_{u}(\Omega')$-obstacle problem and $u+v \in \mathcal{K}^{u}_{u}(\Omega')$, we have
\begin{align*}
\int_{\Omega'} \phi(x, |\nabla u|) \, dx \leq \int_{\Omega'} \phi(x, |\nabla (u+v)|) \, dx
\end{align*}
and therefore $u$ is a local superminimizer in $\Omega$.
\end{proof}
 
\begin{rem}
\label{rem:super-obstacle}
From the previous proof we get also the following result: If a local superminimizer $u$ in $\Omega$ belongs to $W^{1,\phi}(\Omega)$, it is a minimizer of the $\mathcal{K}^{u}_{u}(\Omega)$-obstacle problem.
\end{rem}

Next we prove a comparison principle for the obstacle problem. Strong assumptions are needed to guarantee uniqueness of the minimizer. 
Note that comparison principle also implies uniqueness of the minimizer of the $\mathcal{K}^{f}_{\psi}(\Omega)$-obstacle problem.

\begin{prop}[Comparison principle]
\label{prop:comparison}
Let $\phi \in \Phi_c(\Omega)$ be strictly convex and satisfy (A0), (A1) and (aDec). Let $\psi_1, \psi_2 : \Omega \to [-\infty, \infty)$, $f_1, f_2 \in W^{1,\phi}(\Omega)$ and let $u_1$ and $u_2$ be solutions to the $\mathcal{K}^{f_1}_{\psi_1}(\Omega)$ and $\mathcal{K}^{f_2}_{\psi_2}(\Omega)$-obstacle problems, respectively. If $\psi_1 \leq \psi_2$ almost everywhere in $\Omega$ and $(f_1-f_2)_+ \in W^{1,\phi}_0(\Omega)$, then $u_1 \leq u_2$ almost everywhere in $\Omega$.
\end{prop} 

\begin{proof}
Let $u:=\min \{u_1, u_2\}$ and $h := u_1 -f_1 - (u_2 -f_2) \in W^{1,\phi}_0(\Omega)$. Note that $h_- \in W^{1,\phi}_0(\Omega)$ by  Lemma \ref{lem:min-zero-valued} since both $h$ and the constant function $0$ belong to $W^{1,\phi}_0(\Omega)$ and $\phi$ satisfies (A0), (A1) and (aDec) . Now
\begin{align*}
h \geq \min \{f_2-f_1,h\} &= -\max \{f_1-f_2,-h\} \\
&\geq -\big( \max\{f_1-f_2,0\} + \max\{-h,0\} \big) \\
&= -(f_1-f_2)_+  -h_-.
\end{align*}
Therefore, as $-(f_1-f_2)_+- h_-$ and $h$ belong to $W^{1,\phi}_0(\Omega)$ and $\phi$ satisfies the assumptions in Lemma \ref{lem:u-between-compacts}, we get that $\min\{f_2-f_1, h\} \in W^{1,\phi}_0(\Omega)$. This in turn implies that
\begin{align*}
u-f_1 = \min\{u_2-f_1, u_1-f_1\} = u_2-f_2 + \min \{f_2-f_1,h\} \in W^{1,\phi}_0(\Omega).
\end{align*}
Because $u \geq \psi_1$ almost everywhere in $\Omega$, we see that $u \in \mathcal{K}^{f_1}_{\psi_1}(\Omega)$.

Now let $v:=\max\{u_1,u_2\}$ and $\tilde h :=u_2-f_2 -(u_1-f_1) \in W^{1,\phi}_0(\Omega)$. As before, we get
\begin{align*}
\tilde h \leq \max \{f_1-f_2,\tilde h\} \leq \max\{f_1-f_2,0\} + \max\{\tilde h,0\} = (f_1-f_2)_+ + \tilde h_+.
\end{align*}
By assumptions and Lemma \ref{lem:min-zero-valued}, the functions $(f_1-f_2)_+$, $\tilde h$ and $\tilde h_+ = -\min\{- \tilde h,0\}$ belong to $W^{1,\phi}_0(\Omega)$, so from Lemma \ref{lem:u-between-compacts} we deduce that $\max\{f_1-f_2, \tilde h\} \in W^{1,\phi}_0(\Omega)$. Again,
\begin{align*}
v-f_2 = \max\{u_1-f_2, u_2-f_2\} = u_1-f_1 + \max \{f_1-f_2, \tilde h\} \in W^{1,\phi}_0(\Omega).
\end{align*}
Finally, since $v \geq \psi_2$ almost everywhere in $\Omega$, we see that $v \in \mathcal{K}^{f_2}_{\psi_2}(\Omega)$.

Let $A:=\{u_1 > u_2\}$. Since $u_2$ is a minimizer of the $\mathcal{K}^{f_2}_{\psi_2}(\Omega)$-obstacle problem, we find
\begin{align*}
\int_{\Omega} \phi(x, |\nabla u_2|) \, dx &\leq \int_{\Omega} \phi(x, |\nabla v|)\, dx\\
&= \int_{A} \phi(x, |\nabla u_1|) \, dx + \int_{\Omega \setminus A} \phi(x, |\nabla u_2|) \, dx.
\end{align*}
Now it follows that
\begin{align*}
\int_{A} \phi(x, |\nabla u_2|) \, dx \leq \int_{A} \phi(x, |\nabla u_1|) \, dx
\end{align*}
and therefore
\begin{align*}
\int_{\Omega}\phi(x, |\nabla u|) \,dx &= \int_{A} \phi(x, |\nabla u_2|) \, dx + \int_{\Omega \setminus A} \phi(x, |\nabla u_1|) \, dx \\
&\leq \int_{\Omega} \phi(x, |\nabla u_1|) \, dx.
\end{align*}
Now since $u_1$ is a minimizer of the $\mathcal{K}^{f_1}_{\psi_1}$-obstacle problem, so is $u$. But because $\phi$ is strictly convex and satisfies (A0), the minimizer of the obstacle problem has to be unique \cite[Theorem 7.5]{HarHK16}. Therefore $u_1=u=\min\{u_1, u_2\}$ almost everywhere in $\Omega$ and thus $u_1 \leq u_2$ almost everywhere in $\Omega$.
\end{proof}

The following result is not needed in the rest of the paper, but as it follows quickly from the Comparison principle, we present it for the interested reader.
\begin{prop}
\label{prop:obs-smalleset-super}
Let $\phi \in \Phi_c(\Omega)$ satisfy (A0), (A1), (aDec) and be strictly convex. Let $u$ be a minimizer of the $\mathcal{K}^{f}_{\psi}(\Omega)$-obstacle problem and $v \in \mathcal{K}^{f}_{\psi}(\Omega)$ be a local superminimizer. Then $u \leq v$ almost everywhere in $\Omega$.
\end{prop}

\begin{proof}
Since $v \in \mathcal{K}^{f}_{\psi}(\Omega)$, we have that $v \in W^{1,\phi}(\Omega)$ and by Proposition \ref{prop:obs-is-super} $v$ is the minimizer of the $\mathcal{K}^{v}_{v}(\Omega)$-obstacle problem. Since $u$ and $v$ have the same boundary values in Sobolev sense and $v \geq \psi$ we have from the comparison principle (Proposition \ref{prop:comparison}) that $u \leq v$ almost everywhere in $\Omega$.
\end{proof}

\section{Continuity in the interior}

Later in Section \ref{sec:boundary-continuity} we prove boundary continuity results relating to the solution of the obstacle problem. The proofs rely heavily to similar results inside a domain and the main strategy is to prove irrelevance of the obstacle in most of the points in $\Omega$. At first in this section we collect the relevant results from \cite{HarH_pp16} and formulate them for the obstacle problem and for balls instead of cubes. The original reason for cubes has been to employ Krylov--Safanov covering theorem.

The first lemma corresponds to \cite[Lemma 3.2]{HarH_pp16}, where instead of minimizer of the $\mathcal{K}^{f}_{\psi}(\Omega)$-obstacle problem there is a local quasisubminimizer. All we need to note is that in the proof instead of $-w\eta$ being negative, we have that $v\geq \psi$ if $k \geq \psi$. We also define $A(k,r):= B \cap \{u > k\}$ for any $B \Subset \Omega$ with radius $r$. If $\psi(y) = \infty$ for some $y \in B(x,R)$, we have $A(k,R)=\emptyset$ and the estimate is trivial. 
\begin{lem}[Caccioppoli inequality] 
\label{lem:basic-Caccioppoli}
Let $\phi \in \Phi_w(\Omega)$ satisfy (aDec). Let $u$ be a minimizer of the $\mathcal{K}^{f}_{\psi}(\Omega)$-obstacle problem. Then for all $k \geq \sup_{B(x,R)} \psi$ in $B(x,R)$ we have
\begin{align}
\label{eq:caccioppoli-k}
\int_{A(k,r)} \phi(x,|\nabla (u-k)_+ |) \, dx \leq C \int_{A(k,R)} \phi \left (x, \dfrac{u-k}{R-r} \right ) \, dx
\end{align}
where the constant $C$ depends only on the (aDec) constants of $\phi$.
\end{lem}

Now since $u$ satisfies the previous Caccioppoli inequality, we have the following boundedness result \cite[Proposition 3.3]{HarH_pp16}.

\begin{prop}
\label{prop:u-bdd}
Let $\phi \in \Phi_w(\Omega)$ satisfy (A0), (A1), (aInc)$_p$ and (aDec)$_q$. Suppose that $u \in W^{1,\phi}(\Omega)$ satisfies the Caccioppoli inequality \eqref{eq:caccioppoli-k}. Then there exists $R_0 \in (0, 1)$ such that
\begin{align*}
\esssup_{\frac{1}{2}B} u \leq k_0 + 1 + cR^{-\tfrac{q}{\alpha p}} \left (\int_{2B} \phi(x, (u-k_0)_+ ) \, dx \right )^{\tfrac{1}{p}},
\end{align*}
for every $k_0 \geq \sup_{2B} \psi$ in $2B$, where $B:= B(y,R)$, when $R \in (0, R_0]$ such that $B(y,6R_0) \subset \Omega$. Here $R_0$ is such that $R_0 \leq c(n)$  and $\varrho_{L^{\phi}(B_{6R_0})}(\nabla u) \leq 1$, $\alpha$ is a constant that depends on $n, p$ and $q$, and the constant $c$ depends only on the parameters in assumptions and the dimension n.
\end{prop}

By assuming (A1-$n$) and boundedness of the minimizer instead of assuming (A1) we have the following result \cite[Corollary 3.6]{HarH_pp16}.
\begin{prop}
\label{prop:u-bdd-better}
Let $\phi \in \Phi_w(\Omega)$ satisfy (A0), (A1-n) and (aDec) and suppose that $u$ is locally bounded and satisfies the Caccioppoli inequality \eqref{eq:caccioppoli-k}. Then
\begin{align*}
\esssup_{\frac{1}{2}B} u - k \leq C \left [\left (\int_{2B} (u-k)_+^{q} \, dx \right )^{\tfrac1q} + R\right ]
\end{align*}
when $B:=B(y,R)$ with $R \in (0, R_0]$ such that $B(y,6R_0)\subset \Omega$ and $k\geq \psi(x)$ almost everywhere in $2B$ and $q \in (0,\infty)$. The constant $C$ depends only on the parameters in assumptions (A0), (A1-n) and (aDec), $n$, $R_0$, $\|u\|_{L^{\infty}(B)}$ and $q$. Especially the constant is independent of $R$.
\end{prop}

Next we use the fact that $u$ is also a local superminimizer (Proposition \ref{prop:obs-is-super}) to get an infimum estimate from below \cite[Theorem 4.3]{HarH_pp16}. Since we are aiming for the weak Harnack inequality we need to assume also nonnegativity of the minimizer $u$.

\begin{prop}[The weak Harnack inequality]
\label{prop:weak-harnack}
Let $\phi \in \Phi_w(\Omega)$ satisfy (A0), (A1-n), (aInc) and (aDec). Let $u \in W^{1, \phix}_{\loc}(\Omega)$ be locally bounded nonnegative local (quasi)superminimizer or a locally bounded minimizer of an obstacle problem in $\Omega$. Then there exists an exponent $h>0$ such that
\begin{align*}
\left ( \fint_{B(y,R)} u^{h} \, dx \right )^{1/h} \leq C  \left [\essinf_{B(y, R/2)} u + R\right ]
\end{align*}
for every $R \leq c(n)$ with $B(y,6R) \Subset \Omega$ and $\int_{B(y,6R)} \phi(x, |\nabla u|) \, dx \leq 1$. The constant $C$ depends only on the parameters in the assumptions and $n$.
\end{prop}

The final result we borrow from non-obstacle case is \cite[Theorem 4.4]{HarH_pp16}. It follows directly to our case since a minimizer of the $\mathcal{K}^{f}_{\psi}(\Omega)$-obstacle problem is also a local superminimizer (Proposition \ref{prop:obs-is-super}).

\begin{prop}
\label{prop:obs-lsc}
Let $\phi \in \Phi_w(\Omega)$ satisfy (A0), (A1-n), (aInc) and (aDec). Let $u$ be a locally bounded minimizer of the $\mathcal{K}^{f}_{\psi}(\Omega)$-obstacle problem which is bounded from below and set
\begin{align*}
u^{\ast}(x) := \essliminf_{y \to x} u(y) := \lim_{r\to 0} \essinf_{B(x,r)} u.
\end{align*}
Then $u^{\ast}$ is lower semicontinuous and $u = u^{\ast}$ almost everywhere. 
\end{prop}

The next scheme is to use lower semicontinuous representatives to prove continuity of $u$. The first lemma shows that $u$ can be defined pointwise everywhere. 

\begin{lem}
\label{lem:lsc-fint}
Let $\phi \in \Phi_w(\Omega)$ satisfy (A0), (A1-$n$), (aInc) and (aDec). Assume that $u$ is a locally bounded local superminimizer in $\Omega$. Then
\begin{align*}
u^\ast(x)= \lim_{r\to 0} \fint_{B(x,r)} u\, dy 
\end{align*}
for all $x \in \Omega$.
\end{lem}

\begin{proof}
Fix $x \in \Omega$ and denote $m_r := \essinf_{B(x,r)} u$ when $r$ is small enough to guarantee that $B(x,6r) \Subset \Omega$. As $u$ is assumed to be locally bounded we may assume that $m_r \leq M < \infty$. Note that $m_r$ is a constant when $x$ and $r$ are fixed. Therefore the function $u-m_{4r}$ is a local superminimizer in the set $B(x,r)$ when $x$ and $r$ are fixed. The weak Harnack inequality (Proposition \ref{prop:weak-harnack}) implies
\begin{align*}
0 \leq \fint_{B(x,6r)} (u-m_{4r})^{h}  \, dy 
\leq C\left [(m_{r}-m_{4r}) + r\right ]^{h}.
\end{align*}
Note that by H\"older's inequality we can choose $h\in (0,1]$ in Proposition \ref{prop:weak-harnack}. Since $u$ is bounded, the right-hand side converges to $0$ as $r \to 0$. Therefore we get
\begin{align*}
\lim_{r\to 0} \fint_{B(x,6r)} (u-m_{4r})^{h} \, dy =0.
\end{align*}
Combining this with the fact that $u$ is locally bounded (Proposition \ref{prop:u-bdd-better}) we find that
\begin{align*}
0 &\leq \fint_{B(x,6r)} u-m_{4r} \, dy \leq \fint_{B(x,6r)} (u-m_{4r})^{h} \sup_{y \in B(x,6r)} (u-m_{4r})^{1-h} \, dy \\
&= \sup_{y \in B(x,6r)} (u-m_{4r})^{1-h}\fint_{B(x,6r)} (u-m_{4r})^{h} \, dy \to 0.
\end{align*}
In conclusion
\begin{align*}
\lim_{r\to 0} \fint_{B(x,6r)} u-m_{4r} \, dy = 0.
\end{align*}
Since $u^\ast$ is the lower semicontinuous representative, the previous limit implies
\begin{align*}
u^\ast(x) &= \essliminf_{y \to x} u(y) = \lim_{r \to 0} m_{4r} =\lim_{r\to 0} \fint_{B(x,6r)} m_{4r} \, dy \\
&=\lim_{r \to 0} \fint_{B(x,6r)} m_{4r} + u - m_{4r} \, dy = \lim_{r\to 0} \fint_{B(x,6r)} u \, dy
\end{align*}
for all $x \in \Omega$.
\end{proof}

Finally we can prove the continuity of the minimizer of a $\mathcal{K}^{f}_{\psi}(\Omega)$-obstacle problem in $\Omega$. This proof is a modification of \cite[Theorem 8.29]{BjoB11}. By lower semicontinuously regularized we mean that $u(x)=\essliminf_{y\to x} u(y)$, that is $u=u^{\ast}$.
\begin{thm}
\label{thm:u-is-continuous}
Assume that $\psi: \Omega \to [-\infty, \infty)$ is continuous and $f \in W^{1,\phi}(\Omega)$. Let $\phi \in \Phi_w(\Omega)$ satisfy (A0), (A1), (A1-$n$), (aInc) and (aDec). Let $u$ be a minimizer of the $\mathcal{K}^{f}_{\psi}(\Omega)$-obstacle problem.
Then the lower semicontinuously regularized representative of a minimizer is continuous. 

Moreover, if $\phi$ is convex, then $u$ is a local minimizer (and therefore locally H\"older continuous) in the open set $A=\{x \in \Omega : u(x) > \psi(x)\}$ with boundary values $u$.
\end{thm}

\begin{proof}
Let us denote the lower semicontinuous representative of $u$ still by $u$. To show that $u$ is continuous, we need to prove that 
\begin{align*}
\limsup_{y\to x} u(y)\leq u(x)
\end{align*} 
for all $x \in \Omega$. By local boundedness (Proposition \ref{prop:u-bdd}) and lower semicontinuity this implies that $u$ is real valued and continuous.

Let $x \in \Omega$ and $\varepsilon$ be positive.
By continuity of $\psi$ we can pick a radius $r$ such that $B := B(x,r)\Subset 2B \Subset \Omega$ and $\sup_{2B} \psi \leq \psi(x) + \varepsilon$. Also, the ball $B$ can be chosen to satisfy
\begin{align}
\label{eq:inf-epsilon}
\essinf_{B} u > u(x) - \varepsilon,
\end{align}
as $u$ is finite by the Proposition \ref{prop:u-bdd-better} and it is lower semicontinuous.
Now lower semicontinuity of $u$ and continuity of $\psi$ imply that
\begin{align*}
u(x) = \essliminf_{y \to x} u(y) \geq \essliminf_{y \to x} \psi(y) = \psi(x) \geq \sup_{2B} \psi -\varepsilon.
\end{align*}
Now from Proposition \ref{prop:u-bdd-better} and \eqref{eq:inf-epsilon} we have for $k = u(x) + \varepsilon$, $q=1$ and $B' := B(x, r'), 0< r' < r$,
\begin{align*}
\esssup_{\frac{1}{2}B'} (u-(u(x)+\varepsilon)) &\leq C \fint_{2B'} (u-(u(x)+\varepsilon))_+ \, dy + r' \\
&\leq C \fint_{2B'} (u- (u(x)-\varepsilon))_+ \, dy + r' \\
&= C \fint_{2B'} (u-(u(x)-\varepsilon)) \, dy + r' \\
&= C \left ( \fint_{2B'} u \, dy - u(x) + \varepsilon\right ) + r' .
\end{align*}

From Lemma \ref{lem:lsc-fint} we have
\begin{align*}
u(x) = \lim_{r\to 0} \fint_{B(x,r)} u(y) \, dy.
\end{align*}
Therefore 
\begin{align*}
\esslimsup_{y \to x} u(y) - u(x) - \varepsilon \leq C\varepsilon.
\end{align*}
Thus the claim follows by letting $\varepsilon \to 0^+$.

Next we prove the second claim. We see that $A$ is open by the continuity of $u$ and $\psi$. Since $\phi$ satisfies (aDec), by Lemma \ref{lem:bdd-local-min} it is enough to test the local minimizer with bounded and compactly supported Sobolev functions. Therefore, let $v \in W^{1,\phi}(A)$ be bounded and compactly supported. Since $u$ and $\psi$ are continuous and $u>\psi$ in $A$, there exists $\varepsilon>0$ such that $u \geq \psi + \varepsilon$ in the compact set $\spt v \subset A$. By boundedness of $v$, we can choose $t \in (0,1)$ such that
\begin{align*}
w:=(1-t)u + t(u+v) = u + tv \geq \psi
\end{align*}
in $A$. Therefore $w \in \mathcal{K}^{u}_{\psi}(A)$. Now, since $u$ is a minimizer of the $\mathcal{K}^{u}_{\psi}(A)$-obstacle problem (Lemma \ref{lem:min-local}) and $\phi$ is convex, we see that
\begin{align*}
\int_{A} \phi(x, |\nabla u|) \, dx &\leq \int_{A} \phi(x, |\nabla w|) \, dx \leq \int_{A} \phi(x, (1-t)|\nabla u| + t |\nabla (u+v)|) \, dx \\
&\leq (1-t) \int_{A} \phi(x, |\nabla u|) \, dx + t \int_{A} \phi(x, |\nabla (u+v)|) \, dx. 
\end{align*}
Next we subtract the first term on the right-hand side and divide by $t$ to obtain
\begin{align*}
\int_{A} \phi(x, |\nabla u|) \, dx \leq \int_{A} \phi(x, |\nabla (u+v)|) \, dx.
\end{align*}
Now, since $|\nabla v| =0$ almost everywhere in the set $\{v =0\}$, we get
\begin{align*}
\int_{\{v \not =0\}} \phi(x, |\nabla u|) \, dx &+ \int_{\{v=0\}}\phi(x,|\nabla u|) \, dx = \int_{A} \phi(x, |\nabla u|) \, dx \leq \int_{A} \phi(x, |\nabla (u+v)|) \, dx \\
&= \int_{\{v \not =0\}} \phi(x, |\nabla (u+v)|) \, dx + \int_{\{v=0\}}\phi(x,|\nabla u|) \, dx.
\end{align*}
Subtracting the last term on the right-hand side from both sides, we get
\begin{align*}
\int_{\{v \not=0\}} \phi(x, |\nabla u|) \, dx \leq \int_{\{v \not=0\}} \phi(x, |\nabla (u+v)|) \, dx.
\end{align*}
By Lemma \ref{lem:bdd-local-min}, $u$ is a local minimizer in $A$ and from \cite[Corollary 1.5]{HarHT17} we obtain local H\"older continuity of $u$ in $A$.
\end{proof}

\section{Continuity up to the boundary}
\label{sec:boundary-continuity}

In order to prove the first main theorem, we need to define regular boundary points of a set $\Omega$. In \cite[Theorem 1.1]{HarH_pp16} it was proven that a point is regular if the so called $\phi$-fatness condition is satisfied at $x_0$ and if $\phi$ is regular enough. In Proposition \ref{lem:dens-fat} we prove that the measure density condition \eqref{eq:measure-density} implies $\phi$-fatness when $q<n$.

\begin{defn}
Let $H(f)$ denote the minimizer with boundary values $f \in W^{1,\phi}(\Omega)$. If $g \in C(\partial \Omega)$, then
\begin{align*}
H_g(x) := \sup_{\substack{f\leq g \\ f \text{ is Lipschitz}}} H(f)(x).
\end{align*}

Let $\Omega \subset \R^n$. A point $x \in \partial \Omega$ is called \emph{regular} if
\begin{align*}
\lim_{\substack{y \to x \\ y \in \Omega}} H_f(y) = f(x) 
\end{align*}
for all $f \in C(\partial \Omega)$.
\end{defn}

The next theorem is the main result of \cite{HarH_pp16}. 

\begin{thm}
\label{thm:no-obstacle}
Let $\Omega \subset \R^n$ be bounded and $x_0 \in \partial \Omega$. Let $\phi \in \Phi_c(\R^n)$ be strictly convex and satisfy (A0), (A1), (A1-n), (aInc) and (aDec). If $\Omega$ is locally $\phi$-fat at $x_0$, then $x_0$ is a regular boundary point.
\end{thm}

Most often capacity of balls is somewhat straightforward to compute. This is the case also with $\phi$-capacity, if we assume (aDec), as we have the estimate \cite[Lemma 2.8]{HarH_pp16}
\begin{align}
\label{eq:ball-cap}
c |B|\phi^{-}_{2B}\left ( \tfrac1r\right ) \leq C_{\phi}(B,2B) \leq c |B|\phi^{+}_{2B}\left (\tfrac1r\right ).
\end{align}
It is also noteworthy to mention that upper and lower bounds are comparable when (A1-$n$) is in force.

Next we extend the relation between measure density condition and capacity fatness to generalized Orlicz case. Note that the assumption $q<n$ corresponds to the classical $p$-fatness situation, where it is commonly assumed that $p<n$ since otherwise singleton sets have positive capacity.

In the following proof, we need Poincar\'e inequality for the function $\phi^-$. This can be proven in the almost same way as in \cite[Proposition 6.2.10]{HarH18} with assumptions (A0) and (A1). The necessary modification is to take an equivalent convex $\Phi$-function $\eta$ and use \cite[Lemma 4.3.2]{HarH18} instead of the Key estimate \cite[Theorem 4.3.3]{HarH18}. This has the advantage of not introducing the additive term as in the general Poincar\'e inequality for generalized Orlicz functions. By the assumption (aDec) we can place the constant of equivalence in front of $\phi^-$.

\begin{lem}
\label{lem:dens-fat}
Let $\phi \in \Phi_w(\R^n)$ satisfy (A0), (A1), (A1-$n$) and (aDec)$_q$. If $q<n$ and the measure density condition $\eqref{eq:measure-density}$ is satisfied at $x_0$, then the complement of $\Omega$ is locally $\phi$-fat at $x_0$. 
\end{lem}

\begin{proof}
Denote $E:=B(x_0,r) \setminus \Omega$.

Now with (aDec) and Poincar\'e inequality \cite[Corollary 7.4.1]{HarH18} we estimate
\begin{align*}
\phi_{2B}^{-}\left (\tfrac 1r \right ) |E| &= \int_{E} \phi_{2B}^{-}\left (\tfrac 1r \right ) \, dx \leq \int_{E} \phi_{2B}^{-}\left (\tfrac vr \right ) \, dx \leq L \beta^{-q} \int_{2B} \phi^- \left ( \dfrac{\beta v}{r} \right ) \, dx \\
&\leq  C  \int_{2B} \phi^-(|\nabla v|) \, dx \leq C \int_{2B} \phi(x, |\nabla v|) \, dx.
\end{align*}

Taking infimum over functions $v$ we get
\begin{align}
\label{eq:cap-meas-phi}
C_{\phi}(E,2B) \geq C |E| \phi_{2B}^{-}\left (\tfrac 1r \right ).
\end{align}

As (A1-$n$) implies that $\phi^{+}_{2B}(\tfrac 1r)$ and $\phi^{-}_{2B}(\tfrac 1r)$ are comparable, from \eqref{eq:cap-meas-phi} and the measure density condition \eqref{eq:measure-density} we deduce
\begin{align*}
C_\phi (B(x_0,r) \setminus \Omega, 2B) &\geq c_\ast \, |B(x_0,r) \setminus \Omega| \, \phi_{2B}^{-}\left (\tfrac 1r \right )\\ 
&\geq c_\ast \, |B| \, \phi_{2B}^{-}\left (\tfrac 1r\right ) \\
&\geq c_\ast \, |B|\, \phi^+_{2B}\left (\tfrac 1r \right ) \\
&\geq c_\ast \,  C_{\phi}(B(x_0,r), 2B),
\end{align*}
where the last inequality follows from \eqref{eq:ball-cap}. Thus the capacity fatness condition is satisfied at $x_0$.
\end{proof}

Finally we are ready to prove the continuity of a minimizer up to the boundary.

\begin{proof}[Proof of Theorem \ref{thm:main-result}]
By Lemma \ref{lem:f-above-psi} we can assume that $f \geq \psi$. Let us first show that
\begin{align}
\label{eq:u-leq-f}
\limsup_{x \to x_0} u(x) \leq f(x_0).
\end{align}
Let us denote $D:=\{x \in \Omega : u(x) > f(x)\}$. If $D=\emptyset$, then \eqref{eq:u-leq-f} holds trivially. Let us then suppose that $D$ is not the empty set. If $x_0 \not \in \partial \Omega \cap \partial D$, then there would exist an open set $U \subset \Omega \setminus \overline{D}$ containing $x_0$ and \eqref{eq:u-leq-f} would follow again trivially. Therefore let $x_0 \in \partial \Omega \cap \partial D$. First we need to show that $u-f \in W^{1,\phi}_0(D)$. 

Since $\phi$ satisfies (A0), (A1), (aDec) and $D$ is bounded, $C^\infty(D)\cap W^{1,\phi}(D)$ is dense in $W^{1,\phi}(D)$ \cite[Theorem 6.4.6]{HarH18}. Let us denote $v_j :=  \max\{u-f-\tfrac 1j, 0 \}$ and notice by continuity of $u$ and $f$ that it has compact support in $D$ for every $j$. From \cite[Lemma 3.4]{HarHT17} we have that compactly supported Sobolev--Orlicz functions belong to $W^{1,\phi}_0(D)$, especially $v_j \in W^{1,\phi}_0(D)$ for every $j$. By monotone convergence \cite[Theorem 4.1]{HarHK16}, $u-f-\tfrac 1j$ converges to $u-f$ in $W^{1,\phi}(D)$ and therefore by Lemma \ref{lem:min-max-conv} $(v_j)$ has a subsequence converging to $\max\{u-f,0\} =u-f$ in $W^{1,\phi}(D)$. Since $W^{1,\phi}_0(D)$ is closed, we see that $u-f \in W^{1,\phi}_0(D)$.

Since by assumption, $f \geq \psi$ in $\Omega$, by Theorem \ref{thm:u-is-continuous} $u$ is a local minimizer in $D$ with $u-f \in W^{1,\phi}_0(D)$.
Since $D \subset \Omega$, the capacity fatness condition with respect to $D$ is satisfied at $x_0$: 
\begin{align*}
C_{\phi} (B(x_0,r) \setminus D, B(x_0,2r)) &\geq C_{\phi} (B(x_0,r) \setminus \Omega, B(x_0,2r))\\
&\geq c \, C_{\phi}(B(x_0,r), B(x_0,2r)),
\end{align*}
where the first inequality follows from monotonicity of capacity \cite[(C2) on p. 6]{HarH_pp16}. Now it follows from Theorem \ref{thm:no-obstacle} that $x_0 \in \partial \Omega \cap \partial D$ is a regular boundary point, that is
\begin{align*}
\lim_{\substack{x \to x_0 \\ x \in D}} u(x) = f(x_0).
\end{align*}
Since $u \leq f$ in $\Omega \setminus D$ we get \eqref{eq:u-leq-f}.

It remains to show that
\begin{align}
\label{eq:u-geq-f}
\liminf_{x \to x_0} u(x) \geq f(x_0).
\end{align}
Let $h$ be the unique minimizer with $h-f \in W^{1,\phi}_0(\Omega)$. By the comparison principle (Proposition \ref{prop:comparison}) we have that $h \leq u$ in $\Omega$. Therefore by regularity of $x_0$ we get
\begin{align*}
\liminf_{x \to x_0} u(x) \geq \lim_{x \to x_0} h(x) = f(x_0).
\end{align*}
Together \eqref{eq:u-leq-f} and \eqref{eq:u-geq-f} yield the result.
\end{proof}

\section{Higher integrability of the gradient}

We start by proving two Caccioppoli inequalities: one inside the domain and one near the boundary. The proofs are quite standard and similar usage of test functions can be found from example in \cite{Che16}. Of the assumptions in the following Caccioppoli inequality (A0), (A1) and (aInc) are only to use Sobolev--Poincar\'e inequality for $\psi$, which combines terms involving  $\psi-\psi_{2B}$ and $\nabla \psi$ for simpler result. Compared to the Caccioppoli inequality previously presented in Lemma \ref{lem:basic-Caccioppoli}, now we do not limit ourselves to the positive part of the minimizer and the obstacle appears as an energy rather than a bound for the constant $k$. The second Caccioppoli inequality on the other hand leverages the boundary function rather than the obstacle.

\begin{lem}[Interior Caccioppoli inequality]
\label{lem:interior-Caccioppoli}
Let $\phi\in \Phi_w(\Omega)$ satisfy (A0), (A1), (aInc) and (aDec),
and let $u$ be a minimizer of the $\mathcal{K}^{f}_{\psi}(\Omega)$-obstacle problem where $f, \psi \in W^{1,\phi}(\Omega)$. Then we have
\begin{align}\label{eq:caccioppoli-2}
\fint_{B} \phi(x,|\nabla u|) \,dx \leq C\fint_{2B} \phi \left (x, \dfrac{|u-u_{2B}|}{\diam (2B)} \right ) \, dx + C \fint_{2B} \phi(x, |\nabla \psi|) \, dx + C,
\end{align}
in the ball $B$ with and $2B \subset \Omega$, $\|\nabla \psi\|_{L^\phi(2B)} < 1$ and a constant $C=C(n,\phi)$.
\end{lem}

\begin{proof}
Choose $1\leq s < t \leq 2$. Let $\eta \in C^{\infty}_0(tB)$ be a cut-off function such that $\eta =1$ in $sB$, $0\leq \eta \leq 1$, $\eta =0$ in $2B\setminus tB$ and $|\nabla \eta| \leq \frac{2}{(t-s)r}$. Let $v$ be the following test function 
\begin{align*}
v:=u-u_{2B} -\eta(u-u_{2B} -(\psi-\psi_{2B})).
\end{align*}
First, it needs to be shown that $v$ is an admissible test function for a suitable obstacle problem. Indeed, $v \in \mathcal{K}^{f-u_{2B}}_{\psi-u_{2B}}(\Omega)$ since $v-(f- u_{2B}) \in W^{1,\phi}_0(\Omega)$, because $\eta \in C_0^{\infty}(2B)$, and 
\begin{align*}
v&=(1-\eta) (u-u_{2B}) + \eta (\psi - \psi_{2B})\\
&\geq (1-\eta) (\psi-u_{2B}) + \eta (\psi - u_{2B}) = \psi- u_{2B}
\end{align*}
almost everywhere in $\Omega$ because $u \geq \psi$ almost everywhere in $\Omega$.

A direct calculation yields
\begin{align*}
|\nabla v| \leq (1-\eta)|\nabla u| + \eta |\nabla \psi| + |u-u_{2B} -(\psi-\psi_{2B})| |\nabla \eta|.
\end{align*}
Since $u$ is a minimizer of the obstacle problem $\mathcal{K}^{f}_{\psi}(\Omega)$, we deduce that $u-u_{2B}$ is a minimizer of $\mathcal{K}^{f-u_{2B}}_{\psi-u_{2B}}(\Omega)$ for which $v$ is an admissible test function. Therefore it follows from Lemma \ref{lem:min-local} that
\begin{align*}
\int_{tB} \phi(x, |\nabla u|) \, dx &= \int_{tB} \phi(x, |\nabla(u-u_{2B})|) \, dx \leq  \int_{tB} \phi(x, |\nabla v|) \, dx \\
& \leq  \int_{tB} \phi \big (x, (1-\eta)|\nabla u| + \eta |\nabla \psi| + (u-u_{2B} -(\psi-\psi_{2B})) |\nabla \eta| \big ) \, dx.
\end{align*}
Using (aDec) and the definition of $\eta$ we get 
\begin{align*}
\int_{tB} \phi(x, |\nabla u|) \, dx \leq C \int_{tB} \phi(x, (1-\eta) |\nabla u|) \, dx + C \int_{tB} \phi(x, |\nabla \psi|) \, dx \\
+ C \int_{tB} \phi \left (x, \frac{|u-u_{2B} -(\psi-\psi_{2B})|}{(t-s)r} \right ) \, dx.
\end{align*}
Since $\eta =1$ in $sB$, we see that $\phi(x,(1-\eta)|\nabla u|) =0$ in $sB$. Also, by decreasing the set $tB$ on the left-hand side of the inequality and increasing the set $tB$ on the right-hand side, we get
\begin{align*}
\int_{sB} \phi(x, |\nabla u|) \, dx \leq C \int_{tB\setminus sB} \phi(x, |\nabla u|) \, dx + C \int_{2B} \phi(x, |\nabla \psi|) \, dx \\
+ C \int_{2B} \phi \left (x, \frac{|u-u_{2B} -(\psi-\psi_{2B})|}{(t-s)r} \right ) \, dx.
\end{align*}
Now we use the hole-filling trick by adding $C\int_{sB} \phi(x, |\nabla u|) \, dx$ to both sides of the previous inequality and get $C+1$ of them in the left-hand side while having just constant $C$ on the right-hand side. Now after dividing the inequality by $C+1$ we get a constant $\theta <1$ as the first constant on the right-hand side
\begin{align*}
\int_{sB} \phi(x, |\nabla u|) \, dx \leq \theta \int_{tB} \phi(x, |\nabla u|) \, dx + C \int_{2B} \phi(x, |\nabla \psi|) \, dx \\
+ C \int_{2B} \phi \left (x, \frac{|u-u_{2B} -(\psi-\psi_{2B})|}{(t-s)r} \right ) \, dx.
\end{align*}
Identifying this inequality with the one in iteration Lemma \ref{lem:iteration}, we see after changing to averages that
\begin{align*}
\fint_{B} \phi(x, |\nabla u|) \, dx \leq C \fint_{2B} \phi \left (x, \dfrac{|u-u_{2B}|}{r} + \dfrac{|\psi-\psi_{2B}|}{r} \right ) \, dx + C \fint_{2B} \phi(x, |\nabla \psi|) \, dx.
\end{align*}
As before, we can use (aDec) to obtain
\begin{align*}
\fint_{B} \phi(x, |\nabla u|) \, dx \leq C \fint_{2B} \phi \left (x, \dfrac{|u-u_{2B}|}{r} \right ) \, dx  + C \fint_{2B} \phi\left (x, \dfrac{|\psi-\psi_{2B}|}{r} \right ) \, dx  \\
+ C \fint_{2B} \phi(x, |\nabla \psi|) \, dx.
\end{align*}
Finally using (aDec) and, as $\phi$ satisfies (A0), (A1) and (aInc)$_p$, Sobolev--Poincar\'e inequality (Proposition \ref{pro:Poincare_Sebastian_s}) with $s=1$ we can estimate the term containing $\psi$
\begin{align*}
C \fint_{2B} \phi\left (x, \dfrac{|\psi-\psi_{2B}|}{r} \right ) \, dx \leq C \fint_{2B} \phi(x, |\nabla \psi|) \, dx + C.
\end{align*}
Therefore we get as an interior Caccioppoli inequality
\begin{align*}
\fint_{B} \phi(x, |\nabla u|) \, dx \leq C \fint_{2B} \phi \left (x, \dfrac{|u-u_{2B}|}{r} \right ) \, dx + C \fint_{2B} \phi(x, |\nabla \psi|) \, dx + C.
\end{align*}
Lastly, we use (aDec) to convert from radius to diamater.
\end{proof}

\begin{lem}[Caccioppoli inequality over the boundary]
\label{lem:outerior-Caccioppoli}
Let $\phi \in \Phi_w(\Omega)$ satisfy (aDec) and let $u$ be a minimizer of the $\mathcal{K}^{f}_{\psi}(\Omega)$-obstacle problem where $f, \psi \in W^{1,\phi}(\Omega)$. Assume that there exists a compact set $K \subset \Omega$ such that $f \geq \psi$ in $\Omega \setminus K$ or that $\phi$ satisfies also (A0) and (A1). 
Then we have
\begin{align}
\label{eq:outerior-caccioppoli}
\begin{split}
\dfrac{1}{|B|}\int_{B\cap \Omega} \phi(x, |\nabla u|) \, dx &\leq \dfrac{C}{|2B|} \int_{2B \cap \Omega} \phi \left (x, \dfrac{|u-f|}{\diam(2B)}\right ) \, dx  + \dfrac{C}{|2B|} \int_{2B\cap \Omega} \phi(x, |\nabla f|) \, dx
\end{split}
\end{align}
in the ball $B:=B(y,r)$ with $y \in \Omega$, $2B\setminus \Omega \not = \emptyset$ and $r < \frac{r_0}{4}$, where $r_0 := \dist\{K,\partial \Omega\}$ and constant the $C$ depends only on $n$ and $\phi$.
\end{lem}

\begin{proof}
If $\phi$ satisfies (A0) and (A1), Lemma \ref{lem:f-above-psi} allows us to assume that $f \geq \psi$ and therefore we can take the compact set $K$ as $\emptyset$.
As for the Caccioppoli inequality, we choose $1\leq s<t\leq 2$ and $\eta \in C^{\infty}_0(tB)$ to be a cut-off function such that $\eta =1$ in $sB$, $0\leq \eta \leq 1$, $\eta =0$ in $2B\setminus tB$ and $|\nabla \eta| \leq \frac{2}{(t-s)r}$. This time we use $v:=u-\eta(u-f)$ as a test function. Here we note that $v \in \mathcal{K}^{f}_{\psi}(tB \cap \Omega)$, since $f \geq \psi$ in $\Omega \setminus K$ and the radius $r$ is small enough. Using similar approach as in proof of interior Caccioppoli inequality, we get
\begin{align*}
\int_{tB \cap \Omega} \phi(x, |\nabla u|) \, dx &\leq  \int_{tB \cap \Omega} \phi(x, |\nabla v|) \, dx \\
&\leq \int_{tB\cap \Omega} \phi(x, (1-\eta)|\nabla u| + |u-f||\nabla \eta|+\eta|\nabla f|) \, dx \\
&\leq C \int_{tB \cap \Omega} \phi(x, (1-\eta)|\nabla u|) \, dx + C \int_{tB \cap \Omega} \phi(x, |u-f||\nabla \eta|) \, dx \\
& \quad + C \int_{tB \cap \Omega} \phi(x, |\nabla f|) \, dx.
\end{align*}
Again by decreasing and increasing integration domains and noting that $\eta=1$ in $sB\cap \Omega$, we continue
\begin{align*}
\int_{sB \cap \Omega} \phi(x, |\nabla u|) \, dx &\leq C \int_{(tB \setminus sB) \cap \Omega} \phi(x, |\nabla u|) \, dx + C \int_{2B \cap \Omega} \phi\left (x, \frac{|u-f|}{(t-s)r} \right ) \, dx \\
& \quad + C \int_{2B \cap \Omega} \phi(x, |\nabla f|) \, dx.
\end{align*}
Repeating the hole-filling trick as in the previous Caccioppoli inequality, we get
\begin{align*}
\int_{sB \cap \Omega} \phi(x, |\nabla u|) \, dx &\leq \theta \int_{tB \cap \Omega} \phi(x, |\nabla u|) \, dx + C \int_{2B \cap \Omega} \phi\left (x, \frac{|u-f|}{(t-s)r} \right ) \, dx \\
& \quad + C \int_{2B \cap \Omega} \phi(x, |\nabla f|) \, dx
\end{align*}
and thus repeating the iteration, Lemma \ref{lem:iteration}, we end up with
\begin{align*}
\int_{B \cap \Omega} \phi(x, |\nabla u|) \, dx &\leq  C \int_{2B \cap \Omega} \phi\left (x, \frac{|u-f|}{r} \right ) \, dx + C \int_{2B \cap \Omega} \phi(x, |\nabla f|) \, dx.
\end{align*}
Now we divide by the measure of balls
\begin{align}
\label{eq:boundary-caccioppoli}
\begin{split}
\dfrac{1}{|B|} \int_{B\cap \Omega} \phi(x, |\nabla u|) \, dx &\leq  \dfrac{C}{|2B|} \int_{2B \cap \Omega} \phi\left (x, \frac{|u-f|}{r} \right ) \, dx  + \dfrac{C}{|2B|} \int_{2B \cap \Omega} \phi(x, |\nabla f|) \, dx.
\end{split}
\end{align}
Finally we use (aDec) to change from $r$ to diameter and get the desired Caccioppoli inequality.
\end{proof}

Next we prove the global higher integrability result.

\begin{proof}[Proof of Theorem \ref{thm:main-theorem}]
Let $B:=B(y,r)$ be a ball with $y \in \Omega$ and a radius $r$ satisfying 
\begin{align}
\label{eq:small-u-f}
\|\nabla(u-f)\|_{L^\phi(3B)} + |3B| < \frac{1}{C},  \|\nabla(u-f)\|_{L^{\phi^{1/s}}(3B)} < 1 \quad \text{and} \quad \|\nabla \psi\|_{L^\phi(2B)} < 1,
\end{align}
where $s\leq p$ satisfies the assumptions in the Sobolev--Poincar\'e inequality (Proposition \ref{pro:Poincare_Sebastian_s}) and $C$ is the constant of the same inequality. If $2B \subset \Omega$, from Caccioppoli inequality (Lemma \ref{lem:interior-Caccioppoli}) we have 
\begin{align*}
\fint_{B}\phi(x, |\nabla u|) \, dx \leq C \fint_{2B} \phi \left (x, \dfrac{|u-u_{2B}|}{\diam(2B)} \right ) \, dx + C \fint_{2B} \phi(x, |\nabla \psi|) \, dx + C.
\end{align*}
For the first term on the right-hand side we can use Sobolev--Poincar\'e inequality (Proposition \ref{pro:Poincare_Sebastian_s}) and introduce a constant $s>1$ from \eqref{eq:small-u-f}  such that
\begin{align}
\label{eq:first-SP}
\begin{split}
\fint_{2B} \phi \left (x, \dfrac{|u-u_{2B}|}{\diam(2B)} \right ) \, dx &\leq C \left (\fint_{2B}\phi(x,|\nabla u|)^{1/s} \, dx \right )^{s} + C \\
&\leq C \left (\fint_{3B\cap \Omega} \phi(x, |\nabla u|)^{1/s} \, dx \right )^s + C.
\end{split}
\end{align}

Now if $2B \setminus \Omega \not = \emptyset$, then we use the Caccioppoli inequality over the boundary (Lemma \ref{lem:outerior-Caccioppoli})
\begin{align*}
\dfrac{1}{|B|}\int_{B\cap \Omega} \phi(x, |\nabla u|) \, dx &\leq \dfrac{C}{|2B|} \int_{2B \cap \Omega} \phi \left (x, \dfrac{|u-f|}{\diam(2B)}\right ) \, dx  + \dfrac{C}{|2B|} \int_{2B\cap \Omega} \phi(x, |\nabla f|) \, dx.
\end{align*}

The idea is to use Sobolev--Poincar\'e inequality also to the term involving $u-f$, but this needs some preparation, as there is no integral average on the right-hand side. First we notice that since $u-f \in W^{1, \phi}_0(\Omega)$, it has a zero extension belonging to $W^{1,\phi}(\R^n)$ as in the proof of Lemma \ref{lem:min-local}. This allows us to extend the domain of integration form $2B \cap \Omega$ to $2B$. Second, we note that using (aDec) we can increase the radii of balls
\begin{align}
\label{eq:radii-increasing}
\begin{split}
\fint_{2B} \phi \left (x, \dfrac{|u-f|}{\diam(2B)} \right ) \, dx &\leq \dfrac{|3B|}{|2B|} \fint_{3B} \phi \left (x, \dfrac{3}{2} \dfrac{|u-f|}{\diam(3B)} \right ) \, dx \\
&\leq C \fint_{3B} \phi \left (x, \dfrac{|u-f|}{\diam(3B)}\right ) \, dx.
\end{split}
\end{align} 
Next we choose a ball $\tilde B := \tilde{B}(x_0, r)$, where $x_0 \in 2B \cap \partial \Omega$. It is easily seen that $\tilde B \subset 3B$. Also by appealing to measure density condition \eqref{eq:measure-density}, we see that there exists a constant $\tilde c \in (0,1)$ such that 
\begin{align}
\label{eq:measure-condition-2}
{|\{x \in 3B : u-f=0\}|} \geq |\{x \in \tilde B : u-f=0\}| \geq |\Omega^c \cap \tilde B| \geq c_\ast |\tilde B| \geq \tilde c |3B|.
\end{align}
For brevity, let us denote $v:=u-f$ and $A:=\{x \in 3B : u-f=0\}$. Let us also recall that
\begin{align}
\label{eq:BMO}
\fint_\Omega |g-g_{\Omega'}| \, dy \leq 2\dfrac{|\Omega|}{|\Omega'|} \fint_\Omega |g-g_{\Omega}| \, dy
\end{align}
when $\Omega' \subset \Omega$ has positive measure \cite[Lemma 2.3]{Hur-88}.

Now by \eqref{eq:measure-condition-2} the set $A$ has positive measure and therefore it is meaningful to state that $v_A=0$. With this we can write
\begin{align*}
\fint_{3B} \phi \left ( x, \dfrac{|v|}{\diam(3B)} \right ) \, dx  = \fint_{3B} \phi \left ( x, \dfrac{|v-v_{3B}|+ |v_{3B}-v_A|}{\diam(3B)} \right ) \, dx.
\end{align*}
After an application of (aDec) we get
\begin{align}
\label{eq:s-p-1}
\begin{split}
\fint_{3B} \phi \left ( x, \dfrac{|v|}{\diam(3B)} \right ) \, dx  &\leq C \fint_{3B} \phi \left ( x, \dfrac{|v-v_{3B}|}{\diam(3B)} \right ) \, dx \\
&\quad+ C \fint_{3B}\phi \left ( x, \dfrac{|v_{3B}-v_A|}{\diam(3B)} \right ) \, dx.
\end{split}
\end{align}
The first term on the right-hand side can be estimated with Sobolev--Poincar\'e inequality (Proposition \ref{pro:Poincare_Sebastian_s}) since \eqref{eq:small-u-f} is in force. Let us then use \eqref{eq:BMO} to estimate the last term
\begin{align*}
\fint_{3B}\phi \left ( x, \dfrac{|v_{3B}-v_A|}{\diam(3B)} \right ) \, dx &\leq \fint_{3B} \phi \left (x, \dfrac{\fint_{3_B} |v-v_A| \, dy}{\diam(3B)} \right ) \, dx \\
&\leq C \fint_{3B} \phi \left (x, \dfrac{ 2 \tfrac{|3B|}{|A|}  \fint_{3_B} |v-v_{3B}|\, dy}{\diam(3B)} \right ) \, dx.
\end{align*}
Now by using \eqref{eq:measure-condition-2} and (aDec) we get
\begin{align*}
\fint_{3B}\phi \left ( x, \dfrac{|v_{3B}-v_A|}{\diam(3B)} \right ) \, dx  \leq C \fint_{3B} \phi \left (x, \dfrac{\fint_{3_B} |v-v_{3B}| \, dy}{\diam(3B)} \right ) \, dx.
\end{align*}

From \eqref{eq:small-u-f} we especially have that $\|\nabla v\|_{L^\phi(3B)} < 1$.
 Thus by Sobolev--Poincar\'e inequality (Proposition \ref{pro:Poincare_Sebastian_s}) with $s=1$ and \eqref{eq:small-u-f} we have that
\begin{align*}
\int_{3B} \phi\left (x, \dfrac{|v-v_{3B}|}{\diam(3B)} \right )\, dx \leq C \left [ \int_{3B} \phi(x, |\nabla v|) \, dx +  |3B| \right ] < 1.
\end{align*}
By the unit-ball property \eqref{eq:unit-ball-property}, we see that $\left \|\dfrac{v-v_{3B}}{\diam(3B)}\right \|_{L^\phi(3B)} \leq 1$, so the assumptions of the Jensen type estimate (Lemma \ref{lem:jensen}) are satisfied. Now using it to pull the integral out from the $\phi$ and noticing that outer integral average is redundant, we continue
\begin{align}
\label{eq:after-jensen}
\begin{split}
\fint_{3B}\phi \left ( x, \dfrac{|v_{3B}-v_A|}{\diam(3B)} \right ) \, dx  &\leq C \fint_{3B} \fint_{3B} \phi \left (y, \dfrac{|v-v_{3B}|}{\diam(3B)} \right ) \, dy + 1 \, dx \\
&=C  \fint_{3B} \phi \left (y, \dfrac{|v-v_{3B}|}{\diam(3B)} \right ) \, dy + C. 
\end{split}
\end{align}
Now the last integral is in a form to which the Sobolev--Poincar\'e inequality is applicable and we see that (after the backwards substitution $v = u-f$)
\begin{align*}
\fint_{3B} \phi \left ( x, \dfrac{|u-f|}{\diam(3B)} \right ) \, dx  &\leq C \left ( \fint_{3B} \phi(x, |\nabla (u-f)|)^{1/s} \, dx \right )^{s} + C \\
&= C \left (\dfrac{1}{|3B|}\int_{3B\cap \Omega} \phi(x, |\nabla (u-f)|)^{1/s}\, dx \right )^{s} + C,
\end{align*}
where the equality follows, as $u-f=0$ outside of $\Omega$. Now finishing with triangle inequality, (aDec) and H\"older's inequality we conclude that
\begin{align}
\label{eq:s-p-2}
\begin{split}
\fint_{3B} \phi \left ( x, \dfrac{|u-f|}{\diam(3B)} \right ) \, dx  &\leq C \left (\dfrac{1}{|3B|} \int_{3B \cap \Omega} \phi(x, |\nabla u|)^{1/s} \, dx \right )^{s} \\
& \quad+ C \left ( \dfrac{1}{|3B|} \int_{3B \cap \Omega} \phi(x, |\nabla f|)^{1/s} \, dx \right )^{s} + C \\
&\leq C \left ( \dfrac{1}{|3B|} \int_{3B \cap \Omega} \phi(x, |\nabla u|)^{1/s} \, dx \right )^{s} \\
&\quad +   \dfrac{C}{|3B|} \int_{3B \cap \Omega} \phi(x, |\nabla f|)\, dx + C.
\end{split}
\end{align}

Combining the Caccioppoli inequalities (Lemmas \ref{lem:interior-Caccioppoli} and \ref{lem:outerior-Caccioppoli}), \eqref{eq:first-SP} and \eqref{eq:s-p-2} we have
\begin{align}
\label{eq:s-estimate}
\begin{split}
\dfrac{1}{|B|} \int_{B \cap \Omega} \phi(x,|\nabla u|) \,dx &\leq C \left (\dfrac{1}{|3B|} \int_{3B \cap \Omega} \phi \left (x, |\nabla u| \right )^{1/s} \, dx\right )^{s} \\
&\quad+ \dfrac{C}{|3B|} \int_{3B \cap \Omega} \phi(x, |\nabla \psi|) \, dx 
\\ &\quad + \dfrac{C}{|3B|} \int_{3B\cap \Omega} \phi(x, |\nabla f|) \, dx + C.
\end{split}
\end{align}
Now let
\begin{alignat*}{2}
g:=
    & \begin{aligned} & \begin{cases}
  \phi(x, |\nabla u|), \quad &\text{if } x \in \Omega \\
  0 &\text{if } x \not \in \Omega
  \end{cases}\\
  \end{aligned}
    &, \quad h:=
  \begin{aligned}
  & \begin{cases}
	\phi(x, |\nabla \psi|) + \phi(x, |\nabla f|), \quad &\text{if } x \in \Omega \\
  0 &\text{if } x \not \in \Omega.
  \end{cases} \\
  \end{aligned}
\end{alignat*}
Writing \eqref{eq:s-estimate} with functions $g$ and $h$ we get
\begin{align*}
\fint_{B} g \, dx \leq C \left ( \fint_{3B} g^{1/s} \, dx \right )^{s} + C \fint_{3B} h \, dx + C,
\end{align*}
where $h$ has higher integrability as $\phi(x,|\nabla \psi|), \phi(x,|\nabla f|) \in L^{1+\delta}(\Omega)$.
Now we can use Gehring's lemma, Lemma \ref{lem:Gehring}, which yields a number $\varepsilon>0$ and a constant $C$ such that
\begin{align*}
\fint_{B} \phi(x,|\nabla u|)^{1+ \varepsilon} \, dx &\leq C \bigg[ \left ( \fint_{3B} \phi(x,|\nabla u|) \, dx \right )^{1+ \varepsilon} \\
&\quad + \fint_{3B} \phi(x, |\nabla f|)^{1+\varepsilon} \, dx + \fint_{3B} \phi(x,|\nabla \psi|)^{1+\varepsilon} \, dx +1 \bigg].
\end{align*}
The theorem follows after a covering argument since $\Omega$ is bounded and $\overline{\Omega}$ is compact.
\end{proof}

\section*{Acknowledgements}

I would like to thank Petteri Harjulehto and Peter H\"ast\"o for their insightful comments on the manuscript. This research was partially supported by Turku University Foundation.

\bigskip

\noindent\small{
\textsc{A. Karppinen}}\\
\small{Department of Mathematics and Statistics,
FI-20014 University of Turku, Finland}\\
\footnotesize{\texttt{arttu.a.karppinen@utu.fi}}\\
\end{document}